\documentclass[12pt, oneside]{amsart}
\usepackage{latexsym}
\usepackage{layout}
\usepackage{amsfonts}
\usepackage{amssymb}
\usepackage{epsfig}
\usepackage{color}
\usepackage{cite}

\usepackage[normalem]{ulem}
\usepackage[perpage,para,symbol*]{footmisc}
\usepackage{fancyhdr}
\usepackage{times}
\usepackage{lipsum} 
\newcommand{\R}{\mathbb{R}}

\numberwithin{equation}{section}

\theoremstyle{plain}
\newtheorem{lem}[equation]{Lemma}
\newtheorem{thm}[equation]{Theorem}

\newtheorem{prop}[equation]{Proposition}
\newtheorem{cor}[equation]{Corollary}

\theoremstyle{definition}
\newtheorem{defi}[equation]{Definition}

\theoremstyle{remark}
\newtheorem{remark}[equation]{Remark}

\usepackage{verbatim}
\usepackage[margin=1.15in]{geometry}
\usepackage[hidelinks]{hyperref} 
\hypersetup{
    colorlinks=true,
}

\begin{document}
\title{Extending Erd\H{o}s- Beck's theorem to higher dimensions}
\author{Thao Do$^1$}
\date{March 2017\\$^1$ Department of Mathematics, MIT, Cambridge MA 02139, USA. Email: thaodo@mit.edu }

\maketitle

\begin{abstract}
Erd\H{o}s-Beck theorem states that $n$ points in the plane with at most $n-x$ points collinear define at least $c xn$ lines for some positive constant $c$. It implies $n$ points in the plane define $\Theta(n^2)$ lines unless most of the points (i.e. $n-o(n)$ points) are collinear. 

In this paper, we will present two ways to extend this result to higher dimensions. 
Given a set $S$ of $n$ points in $\mathbb{R}^d$, we want to estimate a lower bound of the number of hyperplanes they define (a hyperplane is defined or spanned by $S$ if it contains $d+1$ points of $S$ in general position). Our first result says the number of spanned hyperplanes is at least $cxn^{d-1}$ if there exists some hyperplane that contains $n-x$ points of $S$ and 
saturated (as defined in Definition \ref{defi_saturated}). Our second result says $n$ points in $\R^d$ define $\Theta(n^d)$ hyperplanes unless most of the points belong to the union of a collection of flats whose sum of dimension is strictly less than $d$.

Our result has application to point-hyperplane incidences and potential application to the point covering problem.

\end{abstract}


\section{Introduction} 


Given a set $S$ of $n$ points in the plane, we say a line $l$ is a spanning line of $S$ (or $l$ is spanned by $S$)  if $l$ contains at least two distinct points of $S$. 
The following theorem was proposed by Erd\H{o}s and proved by Beck in \cite{Beck}:

\begin{thm}\label{Erdos-Beck} [Erd\H{o}s- Beck's theorem, 1983] Any set $S$ of $n$ points in $\R^2$ among which at most $n-x$ points are collinear spans at least $cxn$ lines for some positive constant $c$.
\end{thm}
As a corollary:
\begin{cor}\label{Beck2}
On the plane, for each $\beta\in (0,1)$, there exists $\gamma>0$ depending on $\beta$ such that for any set $S$ of $n$ points, either a line contains $\beta n$ points of $S$, or the number of spanning lines exceeds $\gamma n^2$.
\end{cor}
Indeed, if no line contains $\beta n$ points of $S$ then by theorem \ref{Erdos-Beck}, the number of spanning lines is at least $c(n-\beta n)n= c(1-\beta)n^2$, so we can take $\gamma= c(1-\beta)$.

Erd\H{o}-Beck's theorem is one of the most well-known applications of the celebrated Szemer\'edi-Trotter theorem \cite{ST}, even though in his original proof, Beck does not use this theorem but a weaker version of it. This theorem in turn has many applications to other geometric problems such as in point-hyperplane incidences (see \cite{Eleke}), in various problems involving volumes of tetrahedra (see \cite{Toth-unit-vol}). 

There are two directions to extend Erd\H{o}s-Beck's theorem: to other fields or to higher dimensions. Since Szemer\'edi-Trotter theorem does not hold in finite fields, we only obtain partial results (see for example \cite{IRZ}, \cite{Jones}). As mentioned in section 2.2 in \cite{Eleke}, not much is known in higher dimensions. 
In this paper, we will present two ways to extend this result to $\R^d$ for any $d\geq 3$, the first one resembles theorem \ref{Erdos-Beck} and the second one  resembles corollary \ref{Beck2}. 

From now on, $d$ is some fixed integer $(d\geq 2)$ and $n$ is rather big compared to $d$.
Before stating the first result, we need to define rich and saturated flats. 
Saturated hyperplanes was introduced in \cite{Eleke} in order to obtain some meaningful point-hyperplane incidence bounds. In this paper, we extend this idea to saturated flats.

\begin{defi}\label{defi_saturated} Given a set $S$ of $n$ points  and a $k$-dimensional flat $F$ in $\R^d$.
Let \textbf{$H_S(F)$} denote the number of  $(k-1)$-dimensional hyperplane in $F$ spanned by $S\cap F$. $F$ is called \textbf{$k$-rich} if  $|F\cap S|\geq k$; $F$ is called \textbf{$\gamma$-saturated} if $H_S(F)\geq \gamma |F\cap S|^{d-1}$. 

We say $F$ is saturated if there exists some $\gamma>0$ such that $F$ is $\gamma$-saturated. We say $F$ is rich if there exists some $c>0$ such that $F$ is $c|S|$-rich.
\end{defi}


Our first main result is the following:
\begin{thm}\label{E-B_d_dim}
Assume  $S$ is a set of $n$ points in $\mathbb{R}^d$, and there is some $c_1$-saturated $c_2n$-rich hyperplane $P$, then there is some positive constant $\gamma$ depending on $c_1$ and $c_2$ such that $H_S(\R^d)\ge \gamma xn^{d-1}$ where $x=|S\setminus P|$.
\end{thm}
When $d=2$, since any line is saturated, we get back Erd\H{o}s-Beck's theorem.
To obtain a result similar to corollary \ref{Beck2}, we start with another classical result in \cite{Beck}:


\begin{thm}\label{Beck}
[Beck's theorem, 1983] Given an integer $d\geq 2$, there are constants $\beta_d, \gamma_d$ in $(0,1)$ such that given any set $S$ of $n$ points in $\mathbb{R}^d$, either there exists a hyperplane that contains at least $\beta_d n$ points of $S$ or the number of spanning hyperplanes is at least $\gamma_d n^d$.
\end{thm}
It may not be clear at first glance how this theorem  when $d=2$ is weaker than corollary \ref{Beck2}. To see the difference, let's ask
what is the maximum value of $\beta_d$ so that there is some $\gamma_d$ that satisfies the condition mentioned above? By corollary \ref{Beck2}, in two dimension, any $\beta_2<1$ would work.
However, this is no longer the case in higher dimensions. For example, in $\mathbb{R}^3$, consider two skew lines $l_1, l_2$ and the set $S$ consisting of $n/2$ points on each line (assuming $n$ is even). It is easy to see that any plane contains at most $n/2+1$ points of $S$, yet there are only $n$ spanning (hyper)planes. 

It is proved  in \cite{Lund-Purdy} that two skew lines case is the only obstacle: a set of points in
$\R^3$
such that no more than
$n-x$
of which lie on a plane or any pair of skew lines, determines $\Omega(nk^2)$ planes. This means for any $\beta\in(0,1)$, if no plane or two skew lines contain more than $\beta n$ points then the number of spanning planes is $\Theta(n^3)$.
In this paper we will extend this idea to any dimension $d\geq 2$: 

\begin{thm} \label{thm_d_dim}
For any $0<\beta<1$ there is some constant $\gamma(\beta)$ depending on $d$ and $\beta$ such that for any set $S$ of $n$ points in $\mathbb{R}^d$, either there exists a collection of flats $\{F_1,\cdots,F_k\}$ whose union contains at least $\beta n$ points of $S$ and $\sum_{i=1}^k \dim F_i<d$, or $H_S(\R^d)\ge \gamma(\beta) n^d$.
\end{thm}

\begin{remark} After posting this paper on arXiv, the author learned that a stronger result was proved by Ben Lund in \cite{Lund} five months earlier. Indeed, theorem 2 part (1) in \cite{Lund} implies theorem \ref{thm_d_dim}. However, Lund's method of proof is quite technical and complicated, using projection and induction by the dimension. Our proof is simpler and more intuitive which uses a completely different method. We believe using this result with some work, we can recover Lund's result. Moreover, for many applications such as point-hyperplane incidences and point-cover problem, theorem \ref{thm_d_dim} is enough. 
\end{remark}




In theorem \ref{thm_d_dim},  when $d=2$, we get back corollary \ref{Beck2} since the only possible collection of flat whose sum of dimension less than 2 is a line. Similarly, when $d=3$ the only possible collection of flat whose sum of dimension less than 3 is either a plane or two lines.
Roughly speaking, this theorem implies $n$ points in $\R^d$ spans $\Theta(n^d)$ hyperplanes (or the space $\R^d$ is saturated) unless most points cluster to a collection of flats whose sum of dimension is strictly less than $d$.  This description is satisfactory because if all but $o(n)$ points are outside union of such a collection, we do not expect to get $\Theta(n^d)$ hyperplanes, as shown in the following result:

\begin{prop}\label{prop}
$S$ is a set of $n$ points in $\R^d$. Assume  all but at most $x$ points belong to the union of flats $\{F_1,\dots, F_k\}$ where $\sum_{i=1}^k \dim F_i <d$, then $H_S(\mathbb{R}^d)\leq (x+d)n^{d-1}$.
\end{prop}

We now discuss some applications of our results. An immediate consequence of theorem \ref{thm_d_dim} is a stronger version of  Beck's theorem:

\begin{cor}\label{cor_beta_d}
Given an integer $d\geq 2$, in $\R^d$, for any $\beta_d\in(0,\frac{1}{d-1})$ there is some $\gamma_d$ such that any $n$ points in $\R^d$ defines at least $\gamma n^d$ hyperplanes unless some hyperplane contains at least $\beta_dn$ points. In other words, any $0<\beta_d<\frac{1}{d-1}$ works in theorem \ref{Beck}.
\end{cor}
Indeed, let $\beta=(d-1)\beta_d<1$ and choose $\gamma_d=\gamma(\beta)$ as in theorem \ref{thm_d_dim}. If the number of spanning hyperplanes is less than $\gamma_d n^d$, there is some collection of flats $\{F_1,\dots, F_k\}$ whose union contains at least $\beta n$ points. This implies some flat contains $\ge (\beta/k) n\geq \beta_d n$ points since $k\leq d-1$. Any hyperplane that contains this flat must contain at least $\beta_d n$ points. 

Theorem \ref{thm_d_dim} also has some application in incidence geometry.   
In \cite{Eleke}, Elekes and T\'oth gave a bound for the number of $k$-rich $\gamma$-saturated hyperplanes w.r.t. $n$ points in $\R^d$, which in turns implies a bound for the number of $k$-rich $\alpha$- degenerate hyperplanes where $\alpha$-degenerate flats are defined as followed: An $r$-flat $F$ in $\R^d$ is \textit{$\alpha$-degenerate} for some $0<\alpha\leq 1$ if $F\cap S\neq \emptyset$ and at most $\alpha |F\cap S|$ points of $F\cap S$ lie in an $(r-1)$-flat.

\begin{thm}\label{Eleke-2}[Elekes-T\'oth] Given a set $S$ of $n$ points in $\R^d$. There is some constant  $C(d,\gamma)>0$ such that for any $k$, the number of $k$-rich, $\gamma$-saturated hyperplanes w.r.t. $S$ is at most $C(d,\gamma) (n^dk^{-(d+1)}+n^{d-1}k^{-(d-1)})$.

This, combined with Beck's theorem, implies there are positive constants $\beta_{d-1}$ and $C(d)$ such that for any set $n$ points in $\R^d$, the number of $k$-rich $\beta_{d-1}$-degenerate hyperplanes is at most $C(d)(n^dk^{-(d+1)}+n^{d-1}k^{-(d-1)})$.

When $d=3$ using Erd\H{o}s-Beck's theorem, one obtains a stronger result: for any $\beta\in(0,1)$ there is some constant $C$ such that the number of $k$-rich $\beta$-degenerate planes is at most $C(n^3k^{-4}+n^2k^{-2})$.
\end{thm}

By corollary \ref{cor_beta_d}, the second part of this theorem holds for any $\beta_{d-1}<1/(d-1)$. Moreover, if we redefine $\alpha$-degenerate as:
\begin{defi}
For integers $0<r\leq d$, given a point set $S$ and an $r$-flat $F$ in $\R^d$, we say $F$ is \textbf{$\alpha$-degenerate} for some $0<\alpha\leq 1$ if $F\cap S\neq \emptyset$ and at most $\alpha |F\cap S|$ points of $F\cap S$ lie in union of some flats whose sum of dimension is strictly less than $r$.
\end{defi}
then using theorem \ref{thm_d_dim}, we have a stronger version of theorem \ref{Eleke-2} in any dimension:

\begin{cor}\label{new-eleke} For any $\beta\in (0,1)$, there is some positive constant
$C(d,\beta)$ such that for any set $n$ points in $\R^d$, the number of $k$-rich $\beta$-degenerate hyperplanes is at most $C(d,\beta)(n^dk^{-(d+1)}+n^{d-1}k^{-(d-1)})$.
\end{cor}


This result in turns has some potential application to the point covering problem. Point covering problem is a famous problem in computation geometry which asks for efficient ways to cover $n$ points in space using lines, curves or hyperplanes, hypersurfaces. In \cite{point covering}, the authors use the point-hyperplane bound in theorem \ref{Eleke-2} to derive a good hyperplane covering algorithm. However, it only works in $\R^3$ because of the strong result when $d=3$. Corollary \ref{new-eleke} is one step closer to extend this result to any dimension.  


The structure of the paper is as followed: section 2 is preliminary; section 3 we prove theorem \ref{E-B_d_dim}; in section 4, we prove theorem \ref{thm_d_dim}  in $\R^3$. The  proof of the general case  will be presented in section 5. We conclude with several open questions in section 6. 


\section{Preliminary}

Observe that if we embed $\R^d$ into $\R\mathbb{P}^d$, theorem \ref{E-B_d_dim} and \ref{thm_d_dim} still hold. From now on, we will assume we are working over $\R\mathbb{P}^d$ even if the statement says $\R^d$. The advantage of working over projective spaces is that we do not need to worry about parallel situation. For example, in $\R\mathbb{P}^3$, given any line $l$ and $P$, either $l\subset P$ or $l\cap P$ at exactly a point. This does not hold in $\R^3$ as $l$ can be parallel to $P$. In this case we can say $l\cap P$ at the infinity point. In general,

\begin{lem}\label{projective_space_union_flats}
For any flats $A,B$ in $\R\mathbb{P}^d$, let $\langle A,B\rangle$ denote the span of $A$ and $B$, the smallest flat that contains both $A$ and $B$. Then
$$\dim \langle A,B\rangle =\begin{cases}
\dim A+\dim B+1 & \mbox{if} A\cap B=\emptyset\\  \dim A+\dim B-\dim A\cap B &\mbox{otherwise}
\end{cases}
$$
\end{lem}

\section{Proof of theorem \ref{E-B_d_dim}}
In this section we will prove theorem \ref{E-B_d_dim}. The key idea is to pair each point outside $P$ to a spanning hyperplane of $P$ to form a spanning hyperplane of $\R^d$, then use lemma \ref{lem_fix_point} and Cauchy-Schwartz inequality to take care of the over-counting.

Indeed, since $P$ is $c_1$-saturated and $c_2n$-rich w.r.t. $S$, the number of spanning hyperplanes in $P$ (which are $(d-2)$-dim flats of $\R^d$) is big: $H_S(P)\ge c_1|S\cap P|^{d-1}\ge c_1c_2^{d-1}n^{d-1}$. Let $X=S\setminus P$, the set of points of $S$ outside $P$, then $|X|=x$. Pairing each spanning hyperplanes of $P$ with a point in $X$ we get a spanning hyperplane of our space. If all those hyperplanes are distinct, we expect to see $\sim xn^{d-1}$ of them, exactly what we are trying to prove. However those hyperplanes are usually not distinct. It would be bad if all points of $X$ belong to a line $l$ and all spanning hyperplanes of $P$ pass through $l\cap P$. Fortunately this is not the case, as the following lemma shows:

\begin{lem}\label{lem_fix_point}
For a fixed point $q\in P$ (where $P$ is a $(d-1)$ dim flat), there are at most $n^{d-2}$ spanning hyperplanes of $P$ passing through $q$. 

In particular, for any given a set $S$ of $n$ points in the plane, the number of $S$-spanned lines passing through a fixed point $q$ (not necessarily in $S$) is at most $n$.
\end{lem}

\begin{proof} In any spanning hyperplane $H$ of $P$ that passes through $q$, we can find $d-2$ points of $S\cap H$ so that they together with $q$ form $d-1$ points in general position that spans $H$. Two hyperplanes are distinct only if we can find distinct sets of $d-2$ points. Hence the number of hyperplanes is at most ${n\choose d-2}<n^{d-2}$.
\end{proof} 

Now consider the set of all hyperplanes spanned by a point in $X$ and a spanning hyperplane of $P$: $\mathcal{P}= \{P_1,\dots, P_L\}$ and assume $|P_i\cap X|=a_i$. Then 
$$\sum_{i=1}^L a_i=\#\{(q,H): q\in X, H\in H_S(P)\}\geq c_1c_2^{d-1} xn^{d-1}$$
 Here we abuse the notation $H_S(P)$ to denote the set of all $S$-spanned hyperplanes of $P$. On the other hand, consider
$$J=\#\{(q_1,q_2,P_i): P_i\in \mathcal{P};  q_1,q_2\in X\cap P_i; q_1\neq q_2\}$$
For each choice of $(q_1,q_2)$, the line through them intersect $P$ at some point $q$. For hyperplane $P_i$ in $\mathcal{P}$ that contain $q_1,q_2$, $P_i\cap P$ at some hyperplane of $P$ that contains $q$. By lemma \ref{lem_fix_point}, number of choices for such hyperplanes is at most $n^{d-2}$. Hence $J\leq x^2n^{d-2}$.

On the other hand, for each fixed $P_i$, there are ${a_i\choose 2}$ choices for $(q_1,q_2$). Using Cauchy-Schwartz inequality:
$$J=\sum_{i=1}^L{a_i\choose 2}\geq \frac{1}{3}\sum_{i=1}^L a_i^2-L\geq \frac{1}{3L}\left(\sum_{i=1}^L a_i\right)^2-L\ge 2c\frac{(xn^{d-1})^2}{L}-L$$

where $2c=1/3(c_1c_2)^2$. Rewrite the inequality as $JL+L^2\geq cx^2n^{2d-2}$, we must have either $L^2\geq c x^2n^{2d-2}$ or $LJ\geq c x^2n^{2d-2}$. As $J\leq x^2n^{d-2}$ and $x\leq n$,  in both cases, we would have $L\geq \gamma xn^{d-1}$ for some constant $\gamma$ depending on $c_1$ and $c_2$. Finally it is clear $H_S(\R^d)\geq L\ge \gamma xn^{d-1}$. This completes our proof of theorem \ref{E-B_d_dim}. \qed

\section{Three dimensional case}
In this section we will prove theorem \ref{thm_3_dim}, a special case of our main theorem \ref{thm_d_dim} when $d=3$. We want to prove it separately because its proof is similar, yet much simpler than the general case. We hope that by understanding the proof in this simple case, readers can convince themselves that our strategy works for the general case as well. It is of course totally fine to skip this section and go straight to the next one where the general case's proof is presented.
\begin{thm}\label{thm_3_dim}
For any $\beta\in (0,1)$, there is some constant $\gamma$ depending on $\beta$ such that: for any $n$-points set $S$ in $\mathbb{R}^3$, either there exists a plane contains at least $\beta n$ points of $S$, or there are two skew lines whose union contains at least $\beta n$ points of $S$, or the number of spanning hyperplanes exceeds $\gamma n^3$. 
As a consequence, any $\beta_3\in (0,1/2)$ would work in theorem \ref{Beck}.
\end{thm}
 Assume no plane or two skew lines contains more than $\beta n$ points of $S$, we need to show $H_S(\mathbb{R}^3)\gtrsim_\beta n^3$. Here the notation $\gtrsim_\beta$ means we can put a constant that may depends on $\beta$ right after $\ge$ to make the inequality correct. We will sometimes write $\gtrsim$ when the dependence on $\beta$ is implicit.

By Beck's theorem \ref{Beck}, if no plane contains more than $\beta_3 n$ points, the space is saturated and we are done. So assume there is some plane $P_1$ that contains more than $\beta_3n$ points. If $P_1$ is $\gamma_2$-saturated, theorem \ref{E-B_d_dim} implies $H_S(\mathbb{R}^3)\gtrsim (n-|P_1\cap S|)n^2\gtrsim (1-\beta)n^3$ since no plane contains more than $\beta n$ points. If $P_1$ is not $\gamma_2$-saturated, by theorem \ref{Beck}, some line, say $l_1$, contains more than $\beta_2 |P\cap S|\geq \beta_2\beta_3n$ points of $S$. Excluding this line, there remains at least $(1-\beta)n$ points. We can repeat our argument for those points to find another line $l_2$ that contains at least $(1-\beta)\beta_2\beta_3n$ points. If $l_1,l_2$ belongs to a same plane, then that plane is saturated and contains a portion of points, so we can again apply theorem \ref{E-B_d_dim}. Otherwise $l_1$ and $l_2$ are skew. Because of our assumption, excluding those two lines we still have at least $(1-\beta)n$ points. Repeat the argument one more time, we can find another line $l_3$ that contains at least $(1-\beta)\beta_2\beta_3n$ points and is skew to $l_1$ and $l_2$. We finish our proof by the following lemma:
\begin{lem}\label{3skewlines}
Given 3 lines $l_1,l_2,l_3$ in $\mathbb{R}^3$, pairwise skew, and $|l_i\cap S|\geq c_in$ for $i=1,2,3$. Then $H_S(\mathbb{R}^3)\gtrsim_{c_1,c_2,c_3} n^3$.
\end{lem}
\begin{proof} Heuristically if we pick a point on each line, they will form $\sim n^3$ planes, but those planes may not be distinct. To guarantee distinctness,  we need to pick our points more carefully: for any $p_1\in l_1$, choose $p_2\in l_2$ that does not belong to $\langle p_1,l_3\rangle$, then choose $p_3\in l_3$ which does not belong to $\langle p_1,l_2\rangle$ or $\langle l_1,p_2\rangle$. Now $p_1,p_2,p_3$ spans some plane $H$ such that $H\cap l_i=p_i$. Hence all planes $\langle p_1,p_2,p_3\rangle$ are distinct. So the number of spanning planes is at least $c_1n(c_2n-1)(c_3n-2)\gtrsim n^3$ (since we are not allowed to pick at most 1 point in $l_2$ and at most 2 points in $l_3$).
\end{proof}

\section{Higher dimensions}
The main purpose of this section is to prove theorem\ref{thm_d_dim}. But before we start, we will prove Proposition \ref{prop}, which illustrates that our result is tight.
\\
\\
\textit{Proof of Proposition \ref{prop}:}
For any spanning hyperplane $H$, we can pick $d$ points in $S\cap H$ in general position that generates $H$, call that set $D_H$. One hyperplane may have many generating sets, but two distinct hyperplanes must have distinct ones. Thus $H_S(\R^d)$ is at most the number of generating sets $\{D_H\}$. If $D_H$ contains a point outside $\cup F_i$, there are at most $x$ choices for that point, and ${n\choose d-1}<n^{d-1}$ choices for the remaining $d-1$ points. Therefore in this case the number of distinct $D_H$ is at most $xn^{d-1}$. 
Otherwise, assume $D_H$ contains only points in $\cup F_i$.
Let $a_i$ denote the dimension of $F_i$, and $b_i=|D_H\cap F_i|$. As $\sum b_i\geq d>\sum a_i$, there must exist some $i$ such that $b_i>a_i$, which implies $F_i\subset H$. For each $i\in [k]$, two hyperplanes $H_1$ and $H_2$ that contains $F_i$ are distinguished by the set $D_{H_1}\setminus F_i$ and $D_{H_2}\setminus F_i$. Since $|D_H\setminus F_i|=d-a_i-1$ for each such $H$, there are at most $n^{d-a_i-1}$ spanning hyperplanes that contains $F_i$. Summing together $\#\{D_H\}\le kn^{d-1}$. Therefore, the number of spanning hyperplanes do not exceed $(x+d)n^{d-1}$.
\qed

We now proceed to the proof of theorem \ref{thm_d_dim}. The overall strategy is similar to that of the 3 dimensional cases presented in the previous section: 
Assume any collection of flats whose sum of dimensions less than $d$ does not contain more than $\beta n$ points of $S$, we will show that $H_S(\R^d)\gtrsim n^d$. By Beck's theorem \ref{Beck}, if no hyperplane contains more than $\beta_d n$ points, the space is saturated and we are done. So assume there  is some $\beta_dn$-rich hyperplane $P_1$. If $H_S(P_1)\geq  \gamma_{d-1}|P_1\cap S|^{d-1}\gtrsim n^{d-1}$, we can apply theorem \ref{E-B_d_dim} as now we have a rich saturated hyperplane. Otherwise, by Beck's theorem, $P_1$ contains some $(\beta_d|S\cap P_1|)$-  rich hyperplane $P_2$ (which is of dimension $d-2$). Again by Beck's theorem, either $P_2$ is saturated or it contains some rich hyperplane. Repeating this argument, we end up with a $c_1n$-rich $\gamma_k$-saturated $k$-flat for some constant $c_1$ and $k\leq d-1$. 

By our assumption, this flat contains at most $\beta n$ points. Excluding this flat we are left with at least $(1-\beta)n$ points. Hence we can find another rich and saturated flat. Repeating this argument, we end up with a collection of rich saturated flats $\{F_1,\dots, F_k\}$ whose sum of dimensions is greater than or equal to $d$. 

 We want to prove a result similar to lemma \ref{3skewlines}: a collection of rich saturated flats whose sum is at least $d$ defines $~n^d$ hyperplanes. However, notice that if $\langle F_1,F_2\rangle$, the span of $F_1$ and $F_2$, i.e. the smallest flat that contains both $F_1$ and $F_2$, has dimension less than $\dim F_1+\dim F_2$, by replacing $F_1,F_2$ by $\langle F_1,F_2\rangle$, we obtain another collection of flats whose sum of dimensions decreases. That observation inspires the following definition: 
 
 \begin{defi} Consider a collection of flats $\{F_1,\dots, F_k\}$ in $\mathbb{R}^d$. For any $I\subset [k]:=\{1,\cdots,k\}$, $F_I$ denote the span of $\{F_i\}_{i\in I}$. This collection is called good if $\dim F_{[k]}=d\leq \sum_{i} \dim F_i$ while $\dim F_I \geq \sum_{i\in I} \dim F_i$ for any $I\subsetneq [k]$.  
 \end{defi}

We are now ready to state the generalization of lemma \ref{3skewlines}.
\begin{lem}\label{lem_good_flats}
In $\mathbb{R}^d$, if there are a good collection of flats $\{F_1,\cdots, F_k\}$ each $F_i$ is of dimension $a_i$, $\gamma_{a_i}$-saturated and $c_in$-rich w.r.t. $S$, then $H_S(\mathbb{R}^d)\gtrsim n^d$.
\end{lem}

We will show that this lemma finishes our proof of theorem \ref{thm_d_dim}.

\textit{Proof of theorem \ref{thm_d_dim}:} Recall from the beginning of this section: after applying Beck's theorem many times, we have a collection of rich saturated flats whose sum of dimensions is at least $d$. If this collection is not a good one, which means there is some $I\subset [k]$ so that $\dim F_I<\sum_{i\in I} \dim F_i$. By lemma \ref{lem_good_flats} apply for $d=\dim F_I$, $F_I$ is saturated; clearly $F_I$ is rich. So we can replace $\{F_i\}_{i\in I}$ by their union, $F_I$, to get a new collection of flats whose sum of dimensions decreases. If the sum of dimensions is strictly less than $d$, we repeat our argument using Beck's theorem  to find a new rich saturated flat. If the sum of dimensions is at least $d$ but the collection is still not good, again we can find a way to combine flats $F_I$ as above. This guarantees we will obtain a good collection of flats at some point. Using lemma \ref{lem_good_flats}, our space is saturated. \qed 
\\
\\
It remains to prove lemma \ref{lem_good_flats}, which will be the hardest part of this paper. We encourage readers to read lemma \ref{3skewlines}, a simple version where the good collection of flats are 3 pairwise skew lines, before proceeding any further. If you find some step in the following proof hard to follow,  think about what it means in the case of 3 skew lines. 
\\
\\
\textit{Proof of lemma \ref{lem_good_flats}:} There are two cases: when the sum of dimensions is $d$, and when the sum is strictly greater than $d$. Let us consider  case 1 first, as it is simpler. Case 2 is similar with modification at the last step.
\\
\\
\textbf{Case 1: $\sum_{i=1}^k a_i=d$}

 Heuristically, if we pick $a_i$ points in $S\cap F_i$ for each $i$ to form $\sum_i a_i=d$ points, those points are likely to generate an $S$-spanned hyperplane. There are $\sim n^{a_i}$ choices for points in $F_i$, and thus $\sim n^{\sum a_i}=n^d$ spanning hyperplanes. There are two things that may go wrong: $d$ chosen points may not generate a hyperplane, and the generated hyperplanes may not be distinct. In order to use the saturated of flats $F_i$, instead of picking $a_i$ points, let us pick  an $S$- spanned hyperplanes $P_i$ in each $F_i$. Since $F_i$ is saturated, there are still $\sim n^{a_i}$ choices for $P_i$. 
One way to make sure $\langle P_1,\dots,P_k\rangle=:H$ are distinct is to choose $P_i$ so that $H\cap F_i=P_i$. That motivates the following definition:
\begin{defi}
In $\R^d$, given a good collection of flats $\{F_1,\dots, F_k\}$ w.r.t. $S$, a sequence of flats $\{P_1,\dots, P_k\}$ where $P_i$ is a hyperplane of $F_i$ is a \textbf{nice sequence} if $\langle P_1,\dots, P_k\rangle=H$ is a spanning hyperplane of $\R^d$ and $H\cap F_i=P_i$.
\end{defi}
Clearly each nice sequence generates a distinct spanning hyperplane. Indeed, assume $\{P_1,\dots, P_k\}$ and $\{Q_1,\dots, Q_k\}$ are two nice sequences that generate a same hyperplane $H$. Then $P_i=F_i\cap H=Q_i$ for all $i$, so two sequences are the same.
It remains to show there are $\gtrsim n^d$ distinct nice sequences.
As in lemma \ref{3skewlines}, 
we  shall pick  $P_i$ one at a time in a careful manner. We use the following notations: $F_I:=\langle \{F_i\}_{i\in I}\rangle$; $P_I=\langle \{P_i\}_{i\in I}\rangle$; $a_I=\sum_{i\in I} a_i$ and $[n]=\{1,\dots, n\}$. 

For $s=1,\dots, k$, assume we have picked $P_1,\dots, P_{s-1}$. When $s=1$ it means we have not picked any flat yet. We now choose a spanning hyperplane $P_s$ of $F_s$ such that: $\langle P_s, P_I, F_J\rangle=\langle F_s, P_I, F_J\rangle$ for any $I\subset [s-1], J\subset [k]\setminus I$ that satisfies $\langle P_I, F_J\rangle \cap F_s\neq \emptyset$.
\\
\\
\textit{Claim 1:} 
There are at least $\mu_s n^{a_s}$ choices for such $P_s$ with some positive constant $\mu_s$.
\\
\\
\textit{Proof of claim 1:}
 We count how many spanning hyperplanes in $F_s$ that we cannot pick. For any $I\subset [s-1], J\subsetneq [s+1,\dots, k]$ such that $\langle P_I, F_J\rangle \cap F_l= Q_{I,J}\neq \emptyset$. Any hyperplane $P$ in $F_1$ that does not contain $Q_{I,J}$ satisfies our condition because $\langle Q_{I,J}, P_s\rangle$ is strictly bigger than $P_s$, hence must be $F_s$. The number of $S$-spanned hyperplanes in $F_s$ that contains a fixed point is bounded by $n^{a_s-1}$, hence the number of excluded hyperplanes is $\lesssim_d n^{a_s-1}$. Since $F_s$ is  $\gamma_{a_s}$-saturated, $H_S(F_s)\geq \gamma_{a_s}n^{a_s}$, so for big enough $n$ there remains a portion of $n^{a_s}$ choices for $P_s$.\qed
\\
\\
\textit{Claim 2:} For any $I\subset [k]$ and $J\subset [k]\setminus I$ we have:
\begin{equation}\label{condition_dim}
\dim\langle P_I, F_J\rangle = \begin{cases} 
a_{I}-1 &\mbox{if } J=\emptyset  \\ 
\geq a_{I\cup J} & \mbox{if }  J\neq \emptyset  \end{cases}   
\end{equation}
In particular, the sequence $\{P_i\}_{i\in[k]}$ is a nice one.
\\
\\
\textit{Proof of claim 2:}
We prove \eqref{condition_dim} holds for any $I\subset [s]$ by induction by $s$. When $s=0$, $I=\emptyset$, condition \eqref{condition_dim} becomes $\dim F_J\geq a_J$ which is true as the collection $\{F_i\}$ is good. Assume  \eqref{condition_dim} holds for any $I\subset [s-1]$, we will show that it still holds for any $I\subset [s]$. Clearly we only need to consider the case $s\in I$.

If $\langle P_{I\setminus s},F_J\rangle \cap F_s=\emptyset$, clearly $\langle P_{I\setminus s}, F_J\rangle\cap P_s=\emptyset$, thus $\dim \langle P_I, F_J\rangle= \dim \langle P_{I\setminus s},
F_J\rangle + \dim P_s+1\geq a_{I\cup J\setminus s}+(a_s-1)+1=a_{I\cup J}$ by lemma \ref{projective_space_union_flats}. If on the other hand, $\langle P_{I\setminus s},F_J\rangle \cap F_s\neq \emptyset$, by our choice of $P_s$, $\langle P_I, F_J\rangle = \langle P_{I\setminus s}, F_{J\cup s}\rangle \geq a_{I\cup J}$ as \eqref{condition_dim} holds up to $s-1$.

By induction, $\dim P_{[s-1]} =a_{[s-1]}-1$ and $\dim \langle P_{[s-1]} F_s\rangle\geq a_{[s]}. $ This implies $\dim \langle P_{[s-1]}, F_s\rangle > \dim P_{[s-1]}+\dim F_s$. As a consequence, we must have $P_{[s-1]}\cap F_s=\emptyset$ by lemma \ref{projective_space_union_flats}. Thus $\dim P_{[s]}=\langle P_{[s-1]}, P_s\rangle = \dim P_{[s-1]}+\dim P_s+1= a_{[s]}-1$ as we wished.

Finally we prove $\{P_i\}_i$ is a nice sequence. Let $H:=P_{[k]}$, then $H$ is a hyperplane as $\dim H=a_{[k]}-1=d-1$. For any $i\in [k]$, $H\cap F_i$ has codimension at most 1 in $F_i$, hence either $H\cap F_i=F_i$ or $P_i$. If there is some $i$ such that $H\cap F_i\neq P_i$, then $F_i\subset H$, and $\langle F_i, P_{[k]\setminus i}\rangle\subset H$. However, by \eqref{condition_dim}, $\dim \langle F_i, P_{[k]\setminus i}\rangle\geq a_{[k]}=d$, contradiction. 
\qed
\\
\\
\textbf{Case 2:} The sum of dimensions of good flats is strictly bigger than $d$. We will start with a simple example to inspire the general solution.  

\textbf{Example:} $\{F_1,F_2,F_3\}$ in $\R^8$, each of dimensions three, pairwise non-intersecting and $\langle F_1,F_2,F_3\rangle =\R^5$. Heuristically, we can no longer take a spanning hyperplane in each flat because if we pick 3 generic points of $S$ in each flat to form a plane, those 9 points may span the whole space. Instead, we should take 3 points in $S\cap F_1$, 3 points in $S\cap F_2$ and only 2 points in $S\cap F_3$. As in Case 1, we can find many $S$-spanned planes $P_1\subset F_1$ and $P_2\subset F_2$ such that $\langle P_1,P_2,F_3\rangle=\R^8$ and $\dim \langle P_1,P_2\rangle=5$. By lemma \ref{projective_space_union_flats}, $\langle P_1,P_2\rangle\cap F_3$ at some point $Q$, not necessarily a point of $S$. Let $Q_1:=\langle P_1, F_2\rangle \cap F_3$ and $Q_2:=\langle F_1, P_2\rangle \cap F_3$, then by dimension counting $Q_1,Q_2$ are two lines in $F_3$ and $Q_1\cap Q_2=Q$. 

If there is a plane $P_3$ in $F_3$ that contains $Q$ and an $S$-spanned line $l$ but does not contain $Q_1,Q_2$, then we can check that $H:=\langle P_1,P_2,P_3\rangle$ is a spanned hyperplane of $\R^8$ and $H\cap F_i=P_i$ for $i=1,2,3$. Indeed, $\dim H= \dim \langle P_1,P_2\rangle +\dim P_3=5+2=7$; $H$ is $S$-spanned because we can find 3 points in $S\cap P_1$, 3 points in $S\cap P_2$ and 2 points in $S\cap l$ to form 8 points of $S\cap H$ in general position. To prove $H\cap F_i=P_i$, we prove $H$ does not contain $F_i$. For $i=1,2$, $H$ does not contain $F_i$ because $P_3$ does not contain $Q_i$. For $i=3$, if $F_3\subset H$, $H=\langle P_1, P_2, F_3\rangle =\R^8$ contradiction. 

It remains to count how many choices there are for $P_3$. In $F_3$ which we shall treat as the space $\R^3$, consider the projection map $\pi: F_3 \to F$ where $x\mapsto \langle Q, x\rangle \cap F$. Pick $F$ generic so that $\# \pi(S\cap F_3)\sim n$. Since $F_3$ is saturated, those points define $\sim n^2$ distinct lines. Excluding $\pi(Q_1), \pi(Q_2)$, those statement remain unchanged. The span of $Q$ with any of those $\pi(S\cap F_3)$-spanned lines form a plane $P_3$ that satisfies our condition. There are $\sim n^2$ choices for $P_3$, combine with $\sim n^3$ choices for each $P_1,P_2$ we have $\sim n^8$ spanning hyperplanes.

\textbf{Back to our general case:} $\{F_i\}_{i=1}^k$ is a good collection of flats and $\sum_{i=1}^k a_i=d+x$ for some $x\geq 1$.
Observe that $a_{[k-1]}\leq d-1$ because otherwise we should have considered the collection $\{F_1,\dots, F_{k-1}\}$ instead. As a consequence, $x\leq a_k-1$. We first pick a sequence $\{P_1,\dots, P_{k-1}\}$ of spanning flats as in case 1. Those flats satisfy for any $I\subset[k-1]$:
\begin{equation}
\dim\langle P_I, F_J\rangle = \begin{cases} 
a_{I}-1 &\mbox{if } J=\emptyset  \\ 
\geq a_{I\cup J} & \mbox{if }  J\neq \emptyset, J\subsetneq [k]\setminus I \\
d &\mbox{if} J=[k]\setminus I\end{cases} 
\end{equation}

Now we pick a hyperplane $P_k$ of $F_k$  not necessarily $S$-spanned to form a nice sequence $\{P_1,\dots, P_k\}$, i.e. $H:=\langle P_1,\dots, P_k\rangle$ is an $S$-spanned hyperplane of $\R^d$ and $H\cap F_i=P_i$.
  
Since $\dim P_{[k-1]}=a_{[k-1]}-1$, by lemma \ref{projective_space_union_flats}, $P_{[k-1]}$ intersects $F_k$ at some $(x-1)$-dim flat $Q$. For each $i\in[k-1]$, $\langle F_i, P_{[k-1]\setminus \{i\}}\rangle$ intersects $F_k$ at some $x$-dim flat $Q_i$ which contains $Q$. As in the example, any $P_k$ that contains $Q$ and an $S$-spanned $(a_k-x-1)$-flat in $F_k$ but does not contain $Q_i$ for $i=1,\dots, k-1$ will satisfy our condition. The proof is quite simple and completely similar to that in the example, hence we will omit it here.

In $F_k$ which is equivalent to $\R^{a_k}$, consider a map $\pi$ which is a projection from $Q$ to some generic $(a_k-x)$-dim flat $F$ such that most points of $S\cap F_k$ remain distinct. By dimension counting, each $Q_i$ is projected to a point $q_i$ in $F$. Excluding those $k-1$ points, there remains a portion of $n$ points in $F$.
As $F_k$ is $S$-saturated, we must have $F$ is $\pi(S)$-saturated. In other words, there are $\gtrsim n^{a_k-x}$ flats of dimension $(a_k-x-1)$ that are spanned by $\pi(S)$. The span of $Q$ with each of these flats will generate a hyperplane $P_k$ satisfy our conditions. Hence we have $\gtrsim n^{a_1+\dots+a_{k-1}+a_k-x}=n^d$ distinct hyperplanes in $\mathbb{R}^d$.
This concludes our proof of lemma \ref{lem_good_flats}.\qed

\section{Extension and future work}
In this paper we have generalized the Erd\H{o}s-Beck theorem to higher dimensions as stated in theorems \ref{E-B_d_dim} and \ref{thm_d_dim}. It implies a stronger version of Beck's theorem (corollary \ref{cor_beta_d}) and has some application in incidence geometry. Here are some final thoughts: 
\begin{enumerate}


\item What happens in other fields? Proof of Beck's theorem uses (a weaker version of) Szemer\'edi-Trotter theorem. Since Szemer\'edi-Trotter theorem still holds in complex fields as proved in \cite{Josh} and \cite{Toth}, we can easily extend Beck's theorem to $\mathbb{C}^d$ and all the results in this paper can be extended as well.  As mentioned before, only partial result is known in finite fields. We wonder whether this partial result can be extended to high dimensions in any way using the techniques in this paper.

\item What is the best bound for $\beta_d$ when $d\geq 4$? In corollary \ref{cor_beta_d} we show that any $\beta_d<1/(d-1)$ works in Beck's theorem. This bound is tight when $d=2$ and 3, but it may not be tight for $d\geq 4$. In $\R^4$, if we choose 3 pair-wise skew lines each contains $n/3$ points, some hyperplane will contain two lines, and thus $2n/3$ points. We conjecture that the best bound for $\beta_4$ is $\beta_4<1/2$ by choosing a line and a plane in general position, each contains $n/2$ points. In general, we suspect we can find best bound for $\beta_d$ by carefully analysing all possibilities of  flats whose sum of dimensions is less than $d$.

\item Matroidal version: we can think of the plane as a simple matroid, a line is a $2-flat$, a maximal set of rank 2.  In a simple matroid, most essential properties of points and lines still holds: two lines intersects at at most 1 point, 2 distinct points define at most a line. However, Beck's theorem may not hold in finite fields, we suspect it may not hold in matroids either.

\item In paper \cite{Sharir}, Apfelbaum and Sharir used results about incidences between points and degenerate hyperplanes in \cite{Eleke} to show that if the number of incidences between $n$ points and $m$ arbitrary hyperplanes is big enough, the incidence graph must contain a large complete bipartite subgraph. We wonder if our new version of this result, corollary \ref{new-eleke}, would yield any better result.

\end{enumerate}
\section*{Acknowledgement}
The author would like to thank Larry Guth for suggesting this problem and for his tremendous help and support throughout the project. The author also thanks Ben Yang, Josh Zahl, Hannah Alpert, Richard Stanley and Nate Harman for helpful conversations.

\end{document}